\documentclass[10pt]{amsart}
\usepackage[margin=1in]{geometry} 
\usepackage{graphicx}             
\usepackage{amsmath}              
\usepackage{amsfonts}              
\usepackage{amsthm}     
\usepackage{amssymb}
\usepackage{tikz-cd}       
\usepackage{tikz}   
\usepackage{comment}
\usepackage{color}
\usepackage{hyperref}
\usepackage{pgfplots}
\usetikzlibrary{arrows.meta,bending}

\usepackage[backend=biber]{biblatex}
\newcommand{\heis}{\mathbb{H}}

\newcommand{\hred}{\overline{\heis}}

\newtheorem{theorem}{Theorem}

\newtheorem{lemma}{Lemma}
\newtheorem{remark}{Remark}
\newtheorem{question}{Main Question}
\newtheorem{definition}{Definition}
\newtheorem{corollary}{Corollary}
\newtheorem{proposition}{Proposition}
\newtheorem{conjecture}{Conjecture}

\title[SVD for the X-ray Transforms on the Reduced Heisenberg Group, and a Two-Radius Theorem]{Singular Value Decomposition for the X-ray Transforms on the Reduced Heisenberg Group, and a Two-Radius Theorem}
	\author[S. Flynn]{Steven Flynn}
\address[S. Flynn]{University of Bath, Department of Mathematical Sciences, Bath, BA2 7AY, UK} 
\email{spf34@bath.ac.uk}
\date{May 5th, 2023}
\begin{document}
\maketitle
\begin{abstract}
	We give an explicit Singular Value Decomposition of the sub-Riemannian X-ray transform on the Heisenberg group with compact center. By studying the singular values, we obtain a two-radius theorem for integrals over sub-Riemannian geodesics. We also state intertwining properties of distinguished differential operators. We conclude with a description of ongoing work. 
\end{abstract}
\section{Introduction}
The X-ray transform assigns to a function its integrals over closed geodesics. The Heisenberg group with its sub-Riemannian structure provides a rich example of such integral functionals. Of crucial importance for practical applications of the X-ray Transform is its Singular Value Decomposition  \cite{mishraRangeCharacterizationsSingular2021}.

\section{Definitions}
\subsection{Heisenberg Group}
The Heisenberg group is $\heis=\mathbb{C}\times\mathbb{R}$ with the multiplication law
\begin{align*}
	(x+iy, t)(u+iv, s)=(x+u+i(y+v), t+s+\tfrac{1}{2}(xv-yu))
\end{align*}
The reduced Heisenberg group is the quotient of $\heis$ by a discrete subgroup of the center
\begin{align*}
	\overline{\heis}:=\heis/\Gamma, \quad \Gamma:=\{(0, k\pi)\in \mathbb{C}\times\mathbb{R}: k\in\mathbb{Z}\}.
\end{align*}
Heuristically, the harmonic analysis on the Heisenberg group is controlled by the center $Z(\heis)=\{0+0i \}\times\mathbb{R}$. Having compact center,  analyis of the reduced Heisenberg group resembles that of the circle. 
\subsection{X-ray Transform}
Let $r\in\mathbb{Q}^+$ be a positive rational number. Of interest are curves on $\overline{\heis}$ modeled by
\begin{align}
\gamma_r(s):=\left(\sqrt{r}e^{is/\sqrt{r}}, \tfrac{1}{2}\sqrt{r}s\right)\in \hred\cong \mathbb{C}\times (\mathbb{R}/\pi\mathbb{Z}).\label{geodesics1}
\end{align}
Take for granted that the set of closed unit speed \textit{sub-Riemannian} geodesics \cite{montgomeryTourSubriemannianGeometries2006} on $\hred$  is generated by left translations by $(z, t)\in \overline{\heis}$ of the curves $\gamma_r$ above, for $r\in \mathbb{Q}^+$. The requirement that $r$ is rational is analogous to the requirement that closed geodesics on the flat torus have rational slope.  If $r$ were irrational, the curve $\gamma_r$ would not be closed, so we ignore that case.

Indeed, for $r=a/b$ with $a, b\in\mathbb{N}^+$ coprime,
\begin{align*}
\gamma_r \text{ is a closed curve in }\hred \text{ with period } 2\pi \sqrt{ab}.
\end{align*} 
After one period of $\gamma_r$ with $r=a/b$ for coprime $(a, b)$,
\begin{itemize}
	\item [i] $a\in \mathbb{N}^+$ counts the number of times  $\gamma_r$ winds around the $S^1$ component of $\hred$, 
	\item [ii] $b\in\mathbb{N}^+$ counts the number of counterclockwise rotations made by $\gamma_r$ around its central axis.
\end{itemize}
Call $a$ the \textit{vertical winding number}, and $b$ the \textit{horizontal winding number} of $\gamma_r$ in $\hred$.

\begin{definition} For vertical and horizontal winding numbers $a$ and $b$, and $r=a/b$ the X-ray transform  
	\begin{align*}
	I_{r}: C_c^\infty(\overline{\heis})\to C_c^\infty(\overline{\heis})
	\end{align*}
of a compactly supported smooth function $f\in C_c^\infty(\overline{\heis})$ is
	\begin{align}
	I_{r}f(z, t)=\int_0^{2\pi \sqrt{ab}}f\left((z, t)\gamma_r(s)\right)ds, \quad (z, t)\in \overline{\heis}.\label{xray}
	\end{align}
	We also write $If(r; z, t):=I_{r}f(z, t)$.
\end{definition}
\begin{remark}
	The operator $I$ is the Heisenberg analog of the X-ray transform on a torus. It is related to the X-ray transform on the full Heisenberg group \cite{flynnInjectivityHeisenbergXray2021} by a type of torus-projection operator (as in \cite{ilmavirtaTorusComputedTomography2020}). 
\end{remark}
That $I_r$ is well-defined is straightforward. By the Cauchy Schwartz Inequality and compactness of the domain of integration, $I_r$ extends to a bounded operator on $L^2(\overline{\heis})$. It is a simple matter to extend $I_r$ to other function spaces, but we focus on the $L^2$ theory. 

The primary concern is injectivity of the X-ray transform. Written as in equation \eqref{xray}, the X-ray transform is a special case of the Pompeiu Transform \cite{agranovskyInjectivityPompeiuTransform1994}. However, to the author's knowlegdge, the special case of integration over left-translates a closed sub-Riemannian geodesic has not been considered.

\section{Spectral Decomposition of ${I}_r$, $I_r^*$ and $I_r^* I_r$}
The spectral decomposition of the X-ray transform $I_r$ has a continuous part and a discrete part respecting the  orthogonal decomposition
\begin{align}
	L^2(\overline{\heis})\cong L^2(\mathbb{C})\oplus {^0L^2(\overline{\heis})}. 
\end{align}
The substances in the orthogonal decomposition above are
\begin{align*}
L^2(\mathbb{C})\cong\,  &\{f\in L^2(\overline{\heis}): f(z, t)=f(z, 0),\; \forall(z, t)\in\overline{\heis} \}\\
^0L^2(\overline{\heis}):=\, &\{f\in L^2(\overline{\heis}): \smallint_0^{\pi} f(z, t)dt=0, \; \forall z\in\mathbb{C} \}.
\end{align*}

\begin{proposition}\label{decomp}
	${I}_r$ preserves the orthogonal decomposition above. i.e{,} 
	\begin{align*}
	{{I}_r}: L^2(\mathbb{C})\to L^2(\mathbb{C})
	&&{{I}_r}: {^0}L^2(\overline{\heis})\to {^0}L^2(\overline{\heis}).
	\end{align*}
\end{proposition}
The continuous part of $I_r$ is its restriction to $L^2(\mathbb{C})$, which is essentially the mean value transform on the plane $M^R: L^2(\mathbb{C})\to L^2(\mathbb{C})$,
\begin{align*}
	M^Rf(z):=\frac{1}{2\pi}\int_0^{2\pi}f(z+Re^{i\theta})d\theta, \quad f\in L^2(\mathbb{C}).
\end{align*}

\begin{theorem} The restriction of ${I}_r$ to $L^2(\mathbb{C})$ is a scalar multiple of the mean value transform. In particular
	\begin{align*}
		{I}_r f (z)=2\pi \sqrt{ab}\,M^{\sqrt{r}}f(z), \quad f\in L^2(\mathbb{C}).
	\end{align*}
	Therefore
	\begin{align*}
		I_rf(z)=\frac{\sqrt{ab}}{2\pi}\int_{\mathbb{R}^2} J_0(R|\xi|)\widehat{f}(\xi)e^{iz\cdot \xi}d\xi.
	\end{align*}
\end{theorem}
Here $J_0$ is the zeroth order Bessel function defined in \eqref{bessel}. 
\begin{proof}
	For $f\in L^2(\mathbb{C})$, 
	\begin{align*}
		{I}_r|_{L^2(\mathbb{C})}f(z, t)
		&=\int_0^{2\pi \sqrt{ab}}f\left((z, t)(\sqrt{r}e^{i s/\sqrt{r}}, \tfrac{1}{2}\sqrt{r}s)\right)ds
		=\int_0^{2\pi \sqrt{ab}}f\left(z+e^{i \sqrt{r}s/\sqrt{r}}, 	t+\tfrac{1}{2}\sqrt{r}s+\tfrac{1}{2}\sqrt{r}\textnormal{Im}\left(\overline{z}e^{is/\sqrt{r}}\right)\right)ds\\
		&=\int_0^{2\pi \sqrt{ab}}f(z+\sqrt{r}e^{i s/\sqrt{r}})ds=\sqrt{r}\int_0^{2\pi b}f(z+\sqrt{r}e^{i s})ds=(2\pi b)\sqrt{r} M^{\sqrt{r}}f(z)=2\pi\sqrt{ab}M^{\sqrt{r}}f(z), 
	\end{align*}
	Furthermore
	\begin{align*}
		\int_{\mathbb{R}^2}M^Rg(z)e^{-iz\cdot \xi}dz
		=\int_{\mathbb{R}^2}\frac{1}{2\pi}\int_0^{2\pi}g(z+Re^{i\theta})e^{-iz\cdot\xi}d\theta dz
		=\frac{1}{2\pi}\int_0^{2\pi}e^{iRe^{i\theta}\cdot \xi}d\theta \int_{\mathbb{R}^2}g(z)e^{-iz\cdot \xi}dz=J_0(R|\xi|)\widehat{g}(\xi)
	\end{align*}
	The result follows. 
\end{proof}
We now focus on the Singular Value Decomposition of the discrete part of $I_r$.  
\begin{theorem}[Spectral Decomposition of ${I}_r$ on ${^0L^2(\overline{\heis})}$]
	Fix $r=a/b\in\mathbb{Q}^+$. Then for all $j,k=0, 1, 2,...$
	\begin{align*}
			{I}_{r} \psi_{jk}^{n}=\sqrt{ab}\, \delta_\mathbb{Z}(rn)\, c(rn, j)\psi_{j+r|n|, k}^{n} 
	\end{align*}

	where
	\begin{align*}
	c(m, j)=2\pi \sqrt{\frac{j!}{(j+m)!}}m^{m/2}e^{-(i+1)\pi m/2}L_j^{m}\left(m\right), \quad m\in \mathbb{N}
	\end{align*}
	with
	$c(-m, j)=c(m, j)$ and 
	\begin{align*}
		\delta_{\mathbb{Z}}(m)=
		\begin{cases}
			1 & m\in \mathbb{Z}\\
			0 & m \not\in \mathbb{Z}
		\end{cases}.
	\end{align*}
\end{theorem}
Here $\{\psi_{jk}^n: n\in \mathbb{Z}^*, j, k\in\mathbb{N}\}$ defined in equation \eqref{basis}, is an orthonormal Hilbert Basis of  ${^0 L^2(\overline{\heis})}$. The generalized Laguerre functions are given in \eqref{laguerre}.
\begin{proof} For $n\in \mathbb{Z}^*, j, k\in \mathbb{N}, r\in \mathbb{Q}^+$, and $(z, t)\in \overline{\heis}$,
	\begin{align*}
	{I}_{r} \psi_{jk}^{n}(z, t)
	&=\frac{\sqrt{|n|}}{\pi}{I}_{r} M_{jk}^{n}(z, t)
	=\frac{\sqrt{|n|}}{\pi}\int_0^{2\pi \sqrt{ab}}M_{jk}^{n}\left((z, t)\gamma_{r}(s)\right)ds\\
	&=\frac{\sqrt{|n|}}{\pi}\int_0^{2\pi \sqrt{ab}}\sum_{l=0}^\infty M_{jl}^{n}\left(\gamma_{r}(s)\right) M_{lk}^{n}(z, t) ds, & \text{by } \eqref{product},\\
	&=\sum_{l=0}^\infty\left[ \int_0^{2\pi \sqrt{ab}}M_{jl}^{n}\left(\gamma_{r}(s)\right)  ds\right] \psi_{lk}^{n}(z, t).
	\end{align*}
	Now 
	\begin{align*}
	\int_0^{2\pi \sqrt{ab}}M_{jl}^{n}\left(\gamma_{r}(s)\right) ds
	&=\int_0^{2\pi \sqrt{ab}}M_{jl}^{n}\left(\sqrt{r}e^{i s/\sqrt{r}}, \tfrac{1}{2}\sqrt{r}s\right)ds\\
	&=\sqrt{r}\int_0^{2\pi \sqrt{abr}}M_{jl}^{n}\left(\sqrt{r}e^{is}, \tfrac{1}{2} rs\right)ds, & \text{via } s\mapsto s\sqrt{r}\\
	&=\sqrt{r}\int_0^{2\pi b}e^{i((j-l)+r|n|)s}M_{jl}^{n}\left(\sqrt{r}, 1\right)ds,& \text{by } \eqref{dilation}\\
	&=\sqrt{ab}\int_0^{2\pi }e^{i(b(j-l)+a|n|)s}dsM_{jl}^{n}\left(\sqrt{r}, 1\right), & s\mapsto bs \\
	&=2\pi \sqrt{ab}\, \left[\delta_0\left(b(j-l)+a|n|\right)\right]M^{n}_{jl}(\sqrt{r}, 0).
	\end{align*}
	Thus
	\begin{align*}
	{I}_{r } \psi_{jk}^{n}(z, t)&=2\pi\sqrt{ab}\sum_{l=0}^\infty  \delta(j-l+r|n|)M^{n}_{jl}(\sqrt{r}, 0)\psi_{lk}^{n}(z, t)\\
	&=\delta_\mathbb{Z}(rn)(2\pi \sqrt{ab}) M^{n}_{j, j+|n|r}(\sqrt{r}, 0)\psi_{j+|n|r, k}^{n}(z, t)
	\end{align*}
as desired. 
\end{proof}

\begin{remark}
	In contrast with \cite{flynnInjectivityHeisenbergXray2021}, we see that the spectral decomposition of the X-ray transform involves the ``Bessel spectrum'' (corresponding ot the finite dimensional representations of $\overline{\heis}$). On the full Heisneberg group, this part of the spectrum has Plancherel measure zero. See \cite{follandHarmonicAnalysisPhase1989} or \cite{thangaveluHarmonicAnalysisHeisenberg1998}.  
\end{remark}
For fixed $r\in \mathbb{Q}^+$, the X-ray transform $I_r$ is a convolution operator from $L^2(\overline{\heis})$ to itself. We define the adjoint using the same measure (the Haar measure, which here is the Lebesgue measure) on the domain and target space. 
\begin{theorem}
	The formal adjoint $I_r^*: L^2(\hred)\to L^2(\hred)$ is given by
	\begin{align*}
	I_r^*f(z, t)=\int_0^{2\pi\sqrt{ab}}f\left((z, t)\gamma_r(s)^{-1}\right)ds.
	\end{align*}
\end{theorem}
\begin{remark}
	Note that $\gamma_r(s)^{-1}$ is not a geodesics for the left-invariant metric. In particular $I_r$ is not formally self-adjoint. This fact contrasts with the case for the X-ray transform on the torus with a fixed directional parameter, 
\end{remark}
\begin{corollary}[Spectral Decomposition of ${I}_r$ on ${^0L^2(\overline{\heis})}$]
	Fix $r=a/b\in\mathbb{Q}^+$. Then for all $j,k=0, 1, 2,...$
	\begin{align*}
			{I}^*_{r} \psi_{jk}^{n}=\sqrt{ab}\, \delta_\mathbb{Z}(rn)\, \overline{c(n, j-r|n|)}\psi_{j-r|n|, k}^{n}, 
	\end{align*}
with $c(m, j)$ defined above for $j\geq 0$ and $c(m, j)=0$ for $j<0$. 
\end{corollary}
\begin{corollary}[Spectral Decomposition of $N_r$ on ${^0 L^2(\overline{\heis})}$]
	The normal operator $N_r:= I^*_rI_r$ is diagonalized by the basis $\{\psi_{jk}^n\}$:
	\begin{align*}
			N_r\psi_{jk}^{n}=(ab)\, \delta_\mathbb{Z}(rn)\, |c(rn, j)|^2\psi_{jk}^{n}. 
	\end{align*}
\end{corollary}
\begin{corollary}[Singular Value Decomposition of $I_r$] For $f\in {^0L^2(\overline{\heis})}$, 
	 $I_rf=U_r\circ D_rf$ where
	 \begin{align*}
	 	D_r\psi_{jk}^n&=s(n, j, r)\psi_{jk}^n\\
	 	U_r\psi_{jk}^n&=e^{-\pi r|n|/2}\sigma(r|n|, j)\psi_{j+r|n|, k}^n
	 \end{align*}
 where $\sigma(r|n|, j)={{\rm sgn}}(L_j^{(r|n|)}(r|n|))$ and
 \begin{align*}
 	s(n, j, r)=\sqrt{ab}\, \delta_\mathbb{Z}(rn)\, |c(rn, j)|.
 \end{align*}
\end{corollary}
It is illuminating to see the action of the X-ray Transform on the (reduced) Heisenberg fan, which is the set 
\begin{align*}
	R:=\cup_{j\in \mathbb{N}^*}R_j \quad R_j=R_j^+\cup R_j^-, \quad R_j^\pm=\{(\pm 2n, 2n(2j+1)): n\in\mathbb{N}^+\}.
\end{align*} 
 $R$  is the set of pairs of eigenvalues of $\mathcal{L}$ and $-iT$ (defined in Section \ref{intertwining})  corresponding to the joint eigenvectors $\psi_{jk}^n$ (we omit the part of the Heisenberg fan corresponding to the Bessel spectrum):
The operator $U_r$ in the SVD of $I_r$ acts on the Heisenberg fan according to the picture below. 
\\

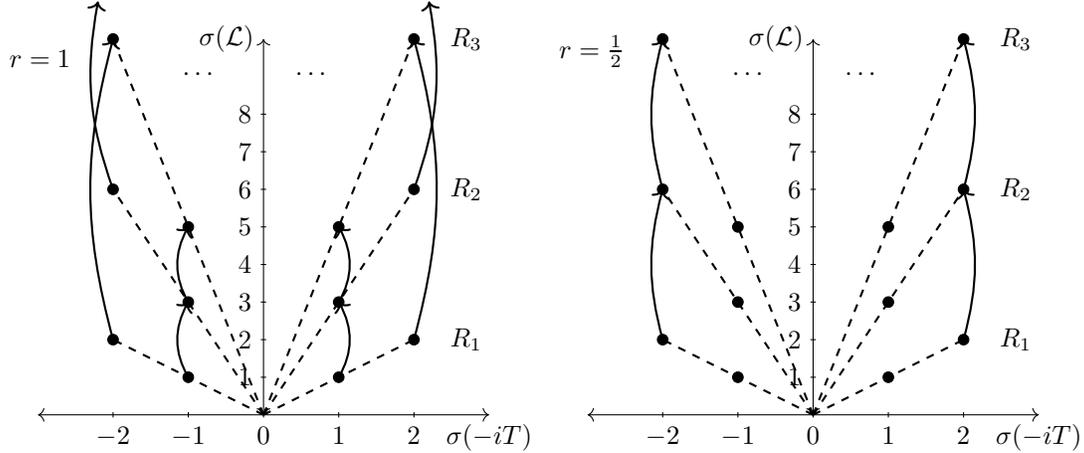
\begin{center}
\begin{tikzpicture}
	\draw[<->] (0,5) node[left]{$\sigma(\mathcal{L})$} -- (0,0);
	\draw[<->] (-3, 0) -- (3,0) node[below]{$\sigma(-iT)$};
	\filldraw (2,1) circle (2pt) node[below left]{};
	\filldraw (2,3) circle (2pt) node[above left] {};
	\filldraw (2,5) circle (2pt) node[below left]{};
	\filldraw (1,1/2) circle (2pt) node[below left]{};
	\filldraw (1,3/2) circle (2pt) node[below left]{};
	\filldraw (1,5/2) circle (2pt) node[below left]{};
	\draw[->,bend right, thick] (1,3/2) to (1,5/2);
	\draw[->,bend right, thick] (1,1/2) to (1,3/2);
	\draw[->,bend right=15, thick] (2,1) to (2,5);
	\draw[->,bend right=15, thick] (2,3) to (2.2,5.5);
	\draw[-,thick, dashed] (0,0) -- (2, 1) node[anchor=south west, right=10] {$R_1$};
	\draw[-,thick, dashed] (0,0) -- (2, 3) node[anchor=south west, right=10] {$R_2$};
	\draw[-,thick, dashed] (0,0) -- (2, 5) node[anchor=south west, right=10] {$R_3$};
	\draw (0.3, 4.4) node[anchor=south west] {$\dots$};
	\draw (-1.2, 4.4) node[anchor=south west] {$\dots$};
	\foreach \x in {-2, -1, 0,1,2}
	\draw (\x cm,1pt) -- (\x cm,-1pt) node[anchor=north] {$\x$};
	\foreach \y in {1,2,3,4, 5, 6, 7, 8}
	\draw (1pt, 0.5*\y cm) -- (-1pt, 0.5*\y cm) node[anchor=east] {$\y$};
		\filldraw (-2,1) circle (2pt) node[below right]{};
	\filldraw(-2,3) circle (2pt) node[above right] {};
	\filldraw (-2,5) circle (2pt) node[below right]{};
	\filldraw (-1,1/2) circle (2pt) node[below right]{};
	\filldraw (-1,3/2) circle (2pt) node[below right]{};
	\filldraw (-1,5/2) circle (2pt) node[below right]{};
	\draw[->,bend left, thick] (-1,3/2) to (-1,5/2);
	\draw[->,bend left, thick] (-1,1/2) to (-1,3/2);
	\draw[->,bend left=15, thick] (-2,1) to (-2,5);
	\draw[->,bend left=15, thick] (-2,3) to (-2.2,5.5);
	\draw[-,thick, dashed] (0,0) -- (-2, 1) node[anchor=south west, left=10] {};
	\draw[-,thick, dashed] (0,0) -- (-2, 3) node[anchor=south west, left=10] {};
	\draw[-,thick, dashed] (0,0) -- (-2, 5) node[anchor=south west, left=10] {};
	\draw (-3.5, 4.5) node[anchor=south west] {$r=1$};
\end{tikzpicture}
\begin{tikzpicture}
	\draw[<->] (0,5) node[left]{$\sigma(\mathcal{L})$} -- (0,0);
	\draw[<->] (-3, 0) -- (3,0) node[below]{$\sigma(-iT)$};
	\filldraw (2,1) circle (2pt) node[below left]{};
	\filldraw (2,3) circle (2pt) node[above left] {};
	\filldraw (2,5) circle (2pt) node[below left]{};
	\filldraw (1,1/2) circle (2pt) node[below left]{};
	\filldraw (1,3/2) circle (2pt) node[below left]{};
	\filldraw (1,5/2) circle (2pt) node[below left]{};
	\draw[->,bend right=15, thick] (2,1) to (2,3);
	\draw[->,bend right=15, thick] (2,3) to (2,5);
	\draw[-,thick, dashed] (0,0) -- (2, 1) node[anchor=south west, right=10] {$R_1$};
	\draw[-,thick, dashed] (0,0) -- (2, 3) node[anchor=south west, right=10] {$R_2$};
	\draw[-,thick, dashed] (0,0) -- (2, 5) node[anchor=south west, right=10] {$R_3$};
	\draw (0.3, 4.4) node[anchor=south west] {$\dots$};
	\draw (-1.2, 4.4) node[anchor=south west] {$\dots$};
	\foreach \x in {-2, -1, 0,1,2}
	\draw (\x cm,1pt) -- (\x cm,-1pt) node[anchor=north] {$\x$};
	\foreach \y in {1,2,3,4, 5, 6, 7, 8}
	\draw (1pt, 0.5*\y cm) -- (-1pt, 0.5*\y cm) node[anchor=east] {$\y$};
	\filldraw (-2,1) circle (2pt) node[below right]{};
	\filldraw (-2,3) circle (2pt) node[above right] {};
	\filldraw (-2,5) circle (2pt) node[below right]{};
	\filldraw (-1,1/2) circle (2pt) node[below right]{};
	\filldraw (-1,3/2) circle (2pt) node[below right]{};
	\filldraw (-1,5/2) circle (2pt) node[below right]{};
	\draw[->,bend left=15, thick] (-2,1) to (-2,3);
	\draw[->,bend left=15, thick] (-2,3) to (-2,5);
	\draw[-,thick, dashed] (0,0) -- (-2, 1) node[anchor=south west, left=10] {};
	\draw[-,thick, dashed] (0,0) -- (-2, 3) node[anchor=south west, left=10] {};
	\draw[-,thick, dashed] (0,0) -- (-2, 5) node[anchor=south west, left=10] {};
	%
	\draw (-3.5, 4.5) node[anchor=south west] {$r=\frac{1}{2}$};
\end{tikzpicture}
\end{center}
In the figures above, dashed lines are rays of the Heisenberg fan (on the full Heisenberg group). Dots are elements of the Heisenberg fan of $\overline{\heis}$, each $(2n, 2|n|(2j+1))$ corresponding to the subspace $\{\psi_{jk}^n\}_{k\in\mathbb{N}}$. Arrows represent the action of $U_r$ for $r=1$ and  $r=1/2$ in the first and second graphic respectively. 

The following functions appear in the expression for the singular values:
\begin{definition}[Diagonal Laguerre Function]
	For $x\in \mathbb{R}$ and $j\in \mathbb{N}$, call $$l_j(x)=L_j^{(x)}(x):=\frac{1}{2\pi i}\oint_C \bigg(\frac{e^{-z/(1-z)}}{1-z}\bigg)^x\frac{dz}{(1-z)z^{j+1}}$$ the \textit{diagonal Laguerre function} of order $j$. Here $C$ is a counterclockwise circle around the origin not encircling the point $z=1$. 
\end{definition}
Next we show that any $L^2$ function on $\overline{\heis}$ is determined by its integrals over geodesics of two $r_1, r_2\in\mathbb{Q}^+$ satisfying a compatibility condition. This type of results is called a Two-Radius Theorem \cite{delsarteLecturesTopicsMean}, but a more appropriate name might be a ``Two-Momenta Theorem," since $\lambda_i:=1/\sqrt{r_j}$ is the momentum of the sub-Riemannian geodesic $\gamma_{r_j}$ \cite{montgomeryTourSubriemannianGeometries2006}. 
\begin{corollary}[Two-Radius Theorem]
	Let $f\in L^2(\overline{\heis})$ with $I_{r_1}f=I_{r_2}f=0$ for some $r_i\in \mathbb{Q}^+$. Then $f=0$ provided that
	\begin{itemize}
		\item[1.] $r_1/r_2$ is not a ratio of roots of diagonal Laguerre functions $l_j(x):=L_j^{(x)}(x)$ for any $j \in \mathbb{N}$
		\item[2.] $\sqrt{r_1/r_2}$ is not a ratio of roots of the Bessel function $J_0$.
	\end{itemize} 
\end{corollary}
\begin{proof}
	For and $n\in \mathbb{Z}^*$ let $s(n, j, r_1)$ and $s(n, j, r_2)$ be the singular values for $\psi_{jk}^n$ corresponding to $I_{r_1}$ and $I_{r_2}$ respectively. Then $f=0$ provided that  $s(n, j, r_1)$ and $s(n, j, r_2)$ are not both zero. We have  $s(n, j, r_1)=s(n, j, r_2)=0$  if and only if $r_jn$ is a zero of the diagonal Laguerre function $l_j$ for $j=1,2$. In this case,  $r_1/r_2$ is a ratio of zeros of Laguerre polynomials. Similarly, $J_0(\sqrt{r_i}|\xi|)$ vanishes if and only if $r_j|\xi|$ is a zero of $J_0$. Then $r_1/r_2$ is a ratio of zeros of $J_0$. 
\end{proof}
\begin{remark}
	There is a lof of room to strengthen this result. For example, one may consider $f$ in more general function spaces. This is the topic of a future work. 
\end{remark}
\section{Intertwining differential operators}\label{intertwining}
Note that the standard left-invariant vector fields on $\heis$ (or on $\overline{\heis}$) are 
\begin{align*}
X:=\partial_x-\frac{1}{2}y\partial_t && Y:=\partial_y+\frac{1}{2}x\partial_t && T:=\partial_t.
\end{align*}
Also, the standard right-invariant vector fields are
\begin{align*}
\tilde{X}:=\partial_x+\frac{1}{2}y\partial_t && \tilde{Y}:=\partial_y-\frac{1}{2}x\partial_t && \tilde{T}:=\partial_t.
\end{align*}
Let $-\mathcal{L}:=X^2+Y^2$, be the left sub-laplacian, and $-\tilde{\mathcal{L}}:=\tilde{X}^2+\tilde{Y}^2$ be the right sub-laplacian. Also let $\square_r=\mathcal{L}+rT^2$ on $\hred$.

\begin{proposition}
	For any right-invariant vector field $\tilde{V}$ and $f\in S(\hred)$ we have 
	\begin{align*}
	{I}_r\left(\tilde{V}f\right)(z, t)=\tilde{V}{I}_r f(z, t).
	\end{align*}
\end{proposition}
\begin{proof}[Proof idea]
	${I}_r$ is a group convolution operator by a compactly supported measure on $\hred$:
	\begin{align*}
	{I}_r f=\mu_r*_\heis f
	\end{align*}
	where
	\begin{align*}
	\int_{\hred}fd\mu_=\int_0^{2\pi \sqrt{ab}}f\left(\gamma_r(s)\right)ds={I}_r f(0).
	\end{align*}
	And for a right-invariant vector field on $\hred$,
	\begin{align*}
	\tilde{V}\left(\mu_r*_\heis f\right)=\mu_r*_\heis \left(\tilde{V}f\right).
	\end{align*}
\end{proof}
\begin{corollary}
	The X-ray transform ${I}_r$ commutes with $T$ and the right sub-laplacian:
	\begin{align*}
	{I}_r \left(Tf\right)=T\left({I}_r f\right) &&
	{I}_r\left(\tilde{\mathcal{L}}f\right)=\tilde{\mathcal{L}}\left({I}_r f\right).
	\end{align*}
\end{corollary}
We use these facts to prove the most interesting intertwining property of ${I}_r$:
\begin{theorem}\label{intertwine}${I}_r$ intertwines the left sublaplacian $\mathcal{L}$ with the operator $\square_r$ on $\hred$:
\begin{align*}
{I}_r\left(\mathcal{L}f\right)=\square_r\left({I}_r f\right), \quad f\in S(\hred).
\end{align*}
\end{theorem}
We first note the helical symmetry of  ${I}_r$.
\begin{lemma}
	Define the rotation map $\mathcal{R}_\theta^*f(z, t)=f(e^{i\theta}z, t)$. Then a straightforward computation gives
	\begin{align*}
		{I}_r \left(\mathcal{R}_{\theta}^*f\right)=\mathcal{R}^*_\theta{I}_r f\left(z, t-\tfrac{1}{2}r\theta\right)
	\end{align*}
	for $f\in C_c^\infty(G)$. Differentiating in $\theta$ yields
$$	I_r(\partial_\theta f)(z, t)=(\partial_\theta-\tfrac{1}{2}rT)I_r f(z, t).$$
\end{lemma}
\begin{proof}[Proof of Theorem \ref{intertwine}]
	A straightforward computation gives
	\begin{align*}
		\tilde{\mathcal{L}}-\mathcal{L}=2\partial_\theta T.
	\end{align*}
	Thus
	\begin{align*}
		{I}_r\left((\mathcal{L}-\tilde{\mathcal{L}})f\right)(z, t)=-2\left(\partial_\theta-\tfrac{1}{2}rT\right)T{I}_r f(z, t),
	\end{align*}
	so that
	\begin{align*}
		{I}_r\left(\mathcal{L}f\right)
		&={I}_r\left(\tilde{\mathcal{L}}f\right)+\left(-2\partial_\theta+r\right)T{I}_r f
		=\tilde{\mathcal{L}}\left({I}_r f\right)+\left(-2\partial_\theta+r\right)T{I}_r f\\
		&=\left(\tilde{\mathcal{L}}-2\partial_\theta T+rT^2\right)\left({I}_r f\right)\
		=\left(\mathcal{L}+rT^2\right)\left({I}_r f\right)\\
		&=:\square_r \left({I}_r f\right)
	\end{align*}
as desired. 
\end{proof}
\begin{remark}
	The functions $\psi_{jk}^{n}$ are joint eigenfunctions of $\mathcal{L}$ and $T$ (see \ref{joint}). Indeed
	\begin{align*}
		\begin{cases}
			\mathcal{L}\psi_{jk}^{n}&=2|n|(1+2j)\,\psi_{jk}^{n}\\
			T\psi_{jk}^{n}&=2in\,\psi_{jk}^{n}.
		\end{cases}
	\end{align*}
	So
	\begin{align*}
		\square_r\psi_{jk}^{n}
		&=2|n|\left(1+2(j-|n|r)\right)\psi_{jk}^{n}.
	\end{align*}
	What's more the functions $\psi_{jk}^{n}$ are eigenfunctions of the right sublaplacian $\tilde{\mathcal{L}}$. Indeed
	\begin{align*}
		\tilde{\mathcal{L}}\psi_{jk}^{n}=2|n|(1+2k)\psi_{jk}^{n}.
	\end{align*}
\end{remark}

\section{Ongoing Work}

\begin{itemize}
	\item[1.] Explicit Inversion formulas. 
	The two-radius theorem implies that a function $f\in L^2(\overline{\heis})$ may be recovered form its X-ray transform $If$. We can therefore write down an inversion formula in terms of the singular values, but it is desirable to have a closed form inversion formula. 
	\item[2.] Chacterize the zeros of the diagonal Laguerre functions. The following conjecture, if true, would imply a one-radius theorem.
	\begin{conjecture}
		The only $(j, n)\in \mathbb{N}\times \mathbb{Z}^*$ for which $l_j(n):=L_j^n(n)=0$ is $(j, n)=(2, 2)$. 
	\end{conjecture}
	\item[2.] Stability estimates. 
	We would like to know if the recovery of $f$ from $If$ is stable. That is, for a suitable choice of Sobolev scale $H^s(\overline{\heis})$ and $H^s(\overline{\heis}\times\mathbb{Q}^+)$ for which 
	\begin{align*}
		\|f\|_{H^s}\leq C \|If\|_{H^{s'}}
	\end{align*}
We expect that, with respect to a natural choice of Sobolev scale based on the intertwining properties above, that $I$ is one  half smoothening. That is $s'=s+\frac{1}{2}$.
	\item Non-Euclidea metrics on the base. 
	The geodesics in \eqref{geodesics1} are determined by lifting the Eudliean metric $g=dx^2+dy^2$  from $\mathbb{R}^2$ to the horizontal distribution of $\overline{\heis}$. What can we say about the X-ray transform associated to the lift of a non-Euclidean metric from $\mathbb{R}^2$ to a sub-Riemannian metric on the Heisenberg group?
	\item Can we obtain a similar explicit spectral decomposition for the X-ray transform on more general Carnot groups such as the Engel group of free Nilpotent groups? Can we state more general two-radius theorems in these cases?
\end{itemize}
\section{Appendix}
The proofs of the following results may be found, for example, in \cite{follandHarmonicAnalysisPhase1989, thangaveluHarmonicAnalysisHeisenberg1998}
\subsection{Entry Functions on the Heisenberg group}
Consider the family of functions, called \textit{matrix coefficients} or \textit{entry functions} defined on the Heisenberg group:
\begin{align}
M^{h}_{jk}(z, t)=
\begin{cases}
\sqrt\frac{k!}{j!} \left(+\sqrt{h}z\right)^{j-k}L_k^{(j-k)}\left(h|z|^2\right)e^{-{h}|z|^2/2}e^{2i{h} t} & j\geq k\\
\sqrt\frac{j!}{k!}\left(-\sqrt{h} \overline{z}\right)^{k-j}L_j^{(k-j)}\left({h}|z|^2\right)e^{-{h}|z|^2/2}e^{2i{h} t} & j \leq k
\end{cases},\label{matrix coefficient}
\end{align}
and $M^{h}_{jk}(z, t)=M^{|{h}|}_{jk}(\overline{z}, -t)$ for ${h}<0$. 

The matrix coefficients are joint eigenfunctions of the left sublaplacian $\mathcal{L}$ and the Reed vector field $T$.
\begin{align}
\begin{cases}
\mathcal{L}M_{jk}^h(z, t)=2|h|\left(1+2j\right)M_{jk}^h(z, t)\\
TM_{jk}^h(z, t)=2ihM_{jk}^h(z, t).
\end{cases}\label{joint}
\end{align}

\begin{lemma}[Useful properties on $M_{jk}^h$]\label{useful}Let $h>0$
	\begin{itemize}
	\item [1)]
	For all $j, k=0, 1, 2, ...$
	\begin{align}
	M_{jk}^h(z, t)&=M_{jk}^1(\sqrt{h}z, ht)\label{dilation}\\
	M_{jk}^{-h}(z, t)&=M_{jk}^h(\overline{z}, -t)\nonumber
	\end{align}
	\item [2)]
	For all $j, j', k, k'=0, 1, 2, ...$
		\begin{align}
	\int_\mathbb{C} M_{jk}^h(z, 0)\overline{M}_{j' k'}^h(z, 0)dz
	=\frac{\pi}{h}\delta_{jj'}\delta_{kk'}. \label{ortho}
	\end{align}
	\item [3)]
	For any $m, n\in\mathbb{Z}^*$ and $j,k=0, 1, 2, ...$
	\begin{align}
	\int_0^{2\pi}M_{jk}^n(e^{ims}, s/2)ds
	&=	2\pi\delta\left(|n|+m(j-k)\right)M_{jk}^n(1, 0). \label{integral}
	\end{align}
	\item [4)] For all $j, k=0, 1, 2, ...$
	\begin{align}
	M_{jk}^h\left((w, s)(z, t)\right)=\sum_{l=0}^\infty M_{jl}^h(z, t)M_{lk}^h(w, s), \quad (w, s), (z, t)\in\heis. \label{product}
	\end{align}
\end{itemize}
\end{lemma}

\begin{definition}
	Define the rescaled functions
	\begin{align}
	\psi_{jk}^{n}:=\frac{\sqrt{|n|}}{\pi}M_{jk}^{n} \label{basis}
	\end{align}
	Since $M_{jk}^{n}(z, t)=e^{2in t}M_{jk}^{n}(z, 0)$ we see that $t$ may be taken mod $\pi$. Thus $\psi_{jk}^{n}$ is a function on $L^2(\overline{\heis})$.
\end{definition}
As a consequence of Lemma \ref{useful}, the set $\{\psi_{jk}^n\}$ is orthonormal. In fact, more is true \cite{thangaveluHarmonicAnalysisHeisenberg1998}:


	\begin{theorem}
	The functions $\{\psi_{jk}^{n}: n\in\mathbb{Z}^*, j,k=0,1, 2, ...\}$ are an orthonormal basis for ${^0L^2(\overline{\heis})}$.
\end{theorem}
\subsection{Special Functions}
\begin{align}
\sum_{j=0}^\infty t^j L_j^{(\alpha)}(x)=\frac{1}{\left(1-t\right)^{\alpha+1}}e^{-xt/(1-t)}, \quad |t|<1.\label{laguerre}
\end{align}
\begin{align}
	J_0(r)=\frac{1}{2\pi}\int_0^{2\pi}e^{ir\cos\theta}d\theta \label{bessel}
\end{align}
\section{Acknowledgments}
Support from the Research Project Grant RPG-2020-037: \textit{Quantum Limits for Sub-elliptic Operators} (PI: V\'{e}ronique Fischer; CoPI: Clotilde Fermanian-Kammerer) is gratefully acknowledged. 
\printbibliography

\end{document}